\documentclass[12pt,a4paper,english,titlepage]{article}
\usepackage{babel,amsmath,amssymb,amsthm}

\newtheorem{theorem}{Theorem}
\newtheorem{lemma}{Lemma}
\newtheorem{claim}{Claim}
\newtheorem*{claim*}{Claim}
\newtheorem*{definition*}{Definition}
\newtheorem*{lemma*}{Lemma}

\newcommand{\EHR}{\mathrm{EHR}}

\newcommand{\cl}{\mathrm{cl}}
\newcommand{\rhomax}{\rho^{\mathrm{max}}}
\newcommand{\x}{{\bf x}}
\newcommand{\y}{{\bf y}}
\newcommand{\defeq}{\mathrel{\mathop:}=}
\newcommand{\z}{{\bf z}}

\begin{document}

\begin{center}{\Large Zero-One Law for random uniform hypergraphs}
\end{center}

\vspace{0.1cm}

\begin{center}{\large A.D. Matushkin}
\end{center}

\vspace{0.5cm}

\section{Introduction}
\label{intro}
In this work limit probabilities of first-order properties of the random $s$-uniform hypergraph in the binomial model $G^{s}(n,p)$ are studied. We give a complete discription of all positive $\alpha$ such that $G^{s}(n,n^{-\alpha})$ obeys Zero-One Law. Moreover, for any rational $\rho\geq 1/(s-1)$ we prove the existence of a strictly balanced $s$-uniform hypergraph with the density $\rho$.\\
   
An {\it $s$-uniform hypergraph}, or {\it $s$-hypergraph}, $G$ is a pair $(V(G),E(G))$ consisting of two sets, the vertices $V(G)$ and edges $E(G)$ of $G$, where each edge $e\in E(G)$ is a set of $s$ elements of $V(G)$. In particular, 2-hypergraph is an ordinary graph.

We consider the binomial model $G^{s}(n,p)$ (see, e.g.,~\cite{1}-\cite{ST}) of random $s$-hypergraph on $n$ vertices with the probability of appearing of an edge $p=p(n)\in[0,1]$. We say that an $s$-hypergraph property $A$ {\it holds almost surely} if
$$	
\lim_{n\rightarrow\infty}\Pr[G^{s}(n,p(n))\models A] = 1
$$
and {\it almost never} if
$$
\lim_{n\rightarrow\infty}\Pr[G^{s}(n,p(n))\models A] = 0.
$$

To the best of our knowledge, one of the most studied class of properties in the sense of almost sure theory is the class of first-order properties. It involves all $s$-hypergraph properties, that could be expressed by first-order formulae. We restrict ourselves by considering only first-order properties in this work.

The random $s$-hypergraph $G^{s}(n,p(n))$ {\it obeys Zero-One Law} (see~\cite{RZ}-\cite{probabilistic_method}) if every first-order property $A$ holds almost surely or almost never.

In 1988 J. Spencer and S. Shelah gave the complete description of all $\alpha>0$, such that the random graph $G(n,n^{-\alpha})$ obeys Zero-One Law. 

\begin{theorem} [J. Spencer, S. Shelah, \cite{SS}]
\label{zero_one_graphs}
Let $p(n)=n^{-\alpha}$, $\alpha>0$. 
\begin{itemize}
\item[(1)] If $\alpha\in\mathbb{Q}\cap(0,1]$ or $\alpha=\frac{k+1}{k}$, $k\in\mathbb{N}$, then $G(n,p(n))$ does not obey Zero-One Law.
\item[(2)] If $\alpha\in(\mathbb{R}\setminus\mathbb{Q})\cap(0,1]$, then $G(n,p(n))$ obeys Zero-One Law.
\item[(3)] If $\alpha>1$, $\alpha\neq\frac{k+1}{k}$, then $G(n,p(n))$ obeys Zero-One Law.
\end{itemize}
\end{theorem}

In this paper, we generalize (1) and (2) of Theorem~\ref{zero_one_graphs} to the case of $s$-hypergraphs. In the proof of the generalization of (1) we exploit the following theorem (that is also proved in this paper): densities of strictly balanced $s$-hypergraphs occupy full spectrum $[1/(s-1),\infty)\cap\mathbb{Q}$. In the proof of the generalization of (2) we use Duplicator's look-ahead strategy, that was proposed by Spencer (see~\cite{Spencer_1991}-\cite{probabilistic_method}).

Clause (3) has already been generalized by N.C. Sandalha and M. Telles. Their result is the following.

\begin{theorem} [N.C. Saldanha, M. Telles, \cite{ST}]
\label{ST}
Let $p=n^{-\alpha}$, $\alpha > s-1$, $\alpha\notin\{s-1+\frac{1}{k},\; k\in\mathbb{N}\}$. Then $G^{s}(n,p)$ obeys Zero-One Law.
\end{theorem}

In Section 2 we review the results on distribution of small subgraphs in the random hypergraph. In Section 3 we state new results. Their proofs are given in Section 4 and Section 6. The last proof exploits the look-ahead strategy which is based on constructions described in Section 5.

\section{Small sub-hypergraphs}
\label{small_hypergraphs}

Let us introduce some notations. For an arbitrary uniform hypergraph $G=(V,E)$ denote $V(G)=V$, $E(G)=E$, $v(G)=|V|$, $e(G)=|E|$. Denote the density $e(G)/v(G)$ of $G$ by $\rho(G)$. We say that $G$ is {\it strictly balanced} if its density is greater than a density of any its proper sub-hypergraph. Obviously, if $G$ is strictly balanced, then it is connected. Moreover, $G$ is strictly balanced if and only if for any its proper connected sub-hypergraph $H$ the inequality $\rho(G)>\rho(H)$ holds.

\begin{theorem} [A. Ruci\'{n}ski, A. Vince, \cite{graphExistence}]
\label{strictly_balanced_graph}
Let $\rho\in\mathbb{Q}$, $\rho\geq 1$. Then there exists a strictly balanced graph with the density~$\rho$.
\end{theorem}

Obviously, if $\rho<1$, then such a strictly balanced graph exists if and only if $\rho=k/(k+1)$ for some $k\in\mathbb{N}$. Indeed, any strictly balanced graph with a density less than $1$ is a tree, and all trees are strictly balanced.

Let $G$ be a strictly balanced $s$-hypergraph with $v$ vertices, $e$ edges and $a$ automorphisms. Denote by $N_{G}$ the number of copies of $G$ in $G^{s}(n,p)$. The following theorem is an obvious generalization of a classical result of Bollob\'{a}s. We believe, this result is well-known. However, we do not know if it has been published. Therefore, we adapted the proof of Bollob\'{a}s's theorem from~\cite{JLR} to the case of uniform hypergraphs.

\begin{theorem}
\label{poisson}
If $p=n^{-v/e}$ then
$$
N_{G}\xrightarrow{d}\mathrm{Pois}\left(\frac{1}{a}\right).
$$
\end{theorem}

\begin{proof}
Consider the factorial moments of $N_{G}$, defined as
$$
{\sf E}\left(N_{G}\right)_{k}={\sf E}\left[N_{G}(N_{G}-1)\ldots(N_{G}-k+1)\right].
$$
We have for $k=1,2,\ldots$
$$
{\sf E}(N_{G})_{k}=\sum\limits_{G_1,\ldots,G_{k}}\Pr\left(I_{G_1}\cdots I_{G_{k}}=1\right)=E'_{k}+E''_{k},
$$
where the summation is over all ordered $k$-tuples of distinct copies of $G$ in $K_{n}$, and $E'_{k}$ is a partial sum where the copies in a $k$-tuple are mutually vertex disjoint. It is easy to see that
$$
E'_{k}\sim ({\sf E}N_{G})^{k}\sim(1/a)^k.
$$

It remains to prove that $E''_{k}=o(1)$. This fact is a consequence of the following claim.

\begin{claim*}
Let $e_{t}$ be the minimum number of edges in a $t$-vertex union of $k$ not mutually vertex disjoint copies of $G$. Then for every $k\geq 2$ and $v\leq t\leq kv-1$ we have $e_{t}>t\rho(G)$.
\end{claim*}

The proof of Claim in the case $s=2$ can be found in~\cite{JLR}, for arbitrary $s$ the proof is the same.
%
%
\end{proof}

Denote by $L_{G}$ the property of containing a copy of $s$-hypergraph $G$. The following theorem is a generalization of a classical result of Erd\H{o}s, R\'{e}nyi~\cite{Erdos} and Bollob\'{a}s~\cite{Bol_small}, stated and proved by A.G. Vantsyan. 

\begin{theorem}[\cite{vantsyan}]
\label{small_distribution}
Fix a finite $s$-hypergraph $G$. If $p(n)\ll n^{-1/\rhomax(G)}$ then
$$
 \lim\limits_{n\rightarrow\infty}\Pr(G^{s}(n,p(n))\models L_{G})=0.
$$
If $p(n)\gg n^{-1/\rhomax(G)}$ then
$$
 \lim\limits_{n\rightarrow\infty}\Pr(G^{s}(n,p(n))\models L_{G})=1.
$$
\end{theorem}

\section{New results}
\label{new_results}
Our objective is to prove a generalization of (1) and (2) of Theorem~\ref{zero_one_graphs} for random uniform hypergraphs.

\subsection{Generalization of (1)}
In order to deal with positive rational $\alpha\leq s-1$ and $\alpha\in\{s-1+\frac{1}{k},\; k\in\mathbb{N}\}$ we extended Theorem~\ref{strictly_balanced_graph} to the case of uniform hypergraphs.

\begin{theorem}
\label{strictly_balanced_hypergraph}
Let $s\in\mathbb{N}$, $s\geq 2$, $\rho\in\mathbb{Q}$. Then there exists a strictly balanced $s$-hypergraph with a density~$\rho$ if and only if $\rho\geq\frac{1}{s-1}$ or $\rho=\frac{k}{1+k(s-1)}$ for some $k\in\mathbb{N}$.
\end{theorem}

As the property of containing a fixed sub-hypergraph $H$ could be expressed by a first order formula, the following theorem is a corollary of Theorems~\ref{poisson} and~\ref{strictly_balanced_hypergraph}.

\begin{theorem}
\label{rational}
If $\alpha\in (0,s-1]\cap\mathbb{Q}$ or $\alpha=s-1+\frac{1}{k}$ for some $k\in\mathbb{N}$ then the random $s$-hypergraph $G^{s}(n,n^{-\alpha})$ does not obey Zero-One Law. 
\end{theorem}

\subsection{Generalization of (2)}

\begin{theorem}
\label{irrational}
If $\alpha$ is a positive irrational number then the random $s$-hypergraph $G^{s}(n,n^{-\alpha})$ obeys Zero-One Law.
\end{theorem}

Thus, according to Theorems~\ref{ST},~\ref{rational} and~\ref{irrational}, the random $s$-hypergraph $G^{s}(n,n^{-\alpha})$ obeys Zero-One Law if and only if $\alpha\in(0,s-1]\setminus{\mathbb{Q}}$ or $\alpha\in(s-1,\infty)\setminus\{s-1+1/k,k\in\mathbb{N}\}$.

The scheme of the proofs is the following. In Section~\ref{proof_rational} we prove Theorem~\ref{strictly_balanced_hypergraph}. In Section~\ref{proof_irrational} we prove Theorem~\ref{irrational}. Main tools for this proof are given in Section~\ref{main_tools}.

\section{Proof of Theorem~\ref{strictly_balanced_hypergraph}}
\label{proof_rational}
We divide the proof into two parts. In the first part we will prove that if $\rho\geq\frac{1}{s-1}$ then there exists a strictly balanced $s$-hypergraph with the density $\rho$. In the second part the case $\rho<\frac{1}{s-1}$ is considered. We will prove that in this case there exists a strictly balanced $s$-hypergraph with the density $\rho$ if and only if $\rho=\frac{k}{1+k(s-1)}$ for some $k\in\mathbb{N}$.

\subsection{Case $\rho\geq\frac{1}{s-1}$}
\label{case1}

\begin{lemma}
\label{lemma1}
Fix rational $\rho>0$. If there exists a strictly balanced $s$-hypergraph ($s\geq 2$) with the density~$\rho$, then there exists a strictly balanced $(s+1)$-hypergraph with the density~$\rho$.
\end{lemma}

\begin{proof} [Proof of lemma~\ref{lemma1}.]
Let $G$ be a strictly balanced $s$--hypergraph with the density $\rho$. Denote by $\widetilde{G}$ the union of two disjoint copies $G_1$, $G_2$ of the hypergraph $G$. Choose one vertex in each copy: $v_1\in V(G_1)$, $v_2\in V(G_2)$. Consider the $(s+1)$--hypergraph $H$ defined as follows. 
$$
V(H)=V(\widetilde{G}),\quad E(H)=\{e\cup\{v_2\},\: e\in E(G_1)\}\cup\{e\cup\{v_1\},\: e\in E(G_2)\}.
$$

Let us prove that $H$ is strictly balanced. Consider an arbitrary proper sub-hypergraph $K\subset H$. Denote $W\defeq\widetilde{G}|_{V(K)}$ --- induced sub-hypergraph of $\widetilde{G}$. As $G$ is strictly balanced we have $\rho(G)\geq\rho(W)$, and the equality holds only if $W=G_1$ or $W=G_2$. Under the definition of $H$ we have $e(K)\leq e(W)$. Moreover, if $W=G_1$ or $W=G_2$ then $e(K)=0$, and the inequality is strict. Therefore, $\rho(K)<\rho(G)=\rho(H)$. Hence $H$ is strictly balanced.\\    
\end{proof}

Now, by Theorem~\ref{strictly_balanced_graph} and Lemma~\ref{lemma1}, it is sufficient to prove Theorem~\ref{strictly_balanced_hypergraph} only for the case $s\geq 3$ and $\frac{1}{s-2}>\rho\geq\frac{1}{s-1}$.

Let $s\geq 3$, $\rho=\frac{m}{n}$, $n=(s-1)m-r$, $r\in\{0,1,\ldots,m-1\}$. We assume that $m\geq 3$. In the cases $m=1$ and $m=2$ consider the fraction $\frac{3m}{3n}$ instead (in that way all fractions would be considered). Consider the set $I=\{\lfloor k\cdot m/r\rfloor,\, k=1,2,\ldots,r\}$ (hereinafter, $\lfloor\cdot\rfloor$ denotes the floor function). Consider the $s$--hypergraph $G$ with $V(G)=\{1,2,\ldots,n\}$ and $E(G)$ defined as follows. There would be $m$ edges in $E(G)$. The first edge is $\{1,2,\ldots,s\}$. Further, if the $k$-th edge is $\{x,x+1,\ldots,x+s-1\}$, then the $(k+1)$-st edge is either $\{x+s-2,x+s-1,\ldots,x+2s-3\}$ if $(k+1)\in I$, or $\{x+s-1,x+s,\ldots,x+2s-2\}$ if $(k+1)\notin I$. While building the edges we identify vertices $1$ and $n+1$, $2$ and $n+2$, etc. (i.e. we suppose that vertices are located on a circle). From the equality $n=(s-1)m-r$ it follows that if $1\notin I$ then the $m$-th edge intersects with the 1-st edge by one vertex, and if $1\in I$ then they intersect by two vertices. Thus every edge has one or two vertices in the intersection with the previous one, and there are exactly $r$ intersections consisting of two vertices.

Let us prove that the defined above $s$--hypergraph $G$ is strictly balanced. Consider an arbitrary proper sub-hypergraph $H\subset G$. We will prove that $\rho(H)<\rho(G)$. We may assume without loss of generality that $H$ is connected, because its density is not greater than the maximum of densities of its connected components. In this case, $V(H)$ is a set of neighbouring (going in a row) vertices of the hypergraph $G$. We may assume that $H=G|_{V(H)}$, because adding some edges to $H$ increases its density. Let the edges of $H$ have numbers $i$, $i+1$, \dots, $j$, where $1\leq i\leq j\leq 2m$. We may assume that if the intersection of the first or the last edge (edge with the number $i$ or $j$ respectively) and the next one consists of two vertices. Indeed, if that intersection consists of one vertex and $\rho(H)\geq\rho(G)$, then we delete this first or last edge and $(s-1)$ its vertices, that are not in the intersection and get a hypergraph with a density not less than $\rho(G)$. Thus we may assume that both the first and the last intersections contain two vertices. Let $i=[km/r]-1$, $j=[lm/r]$ ($1\leq k\leq l\leq 2r$). Then $e(H)=[lm/r]-[km/r]+2$, $v(H)=e(H)s-(e(H)-1)-(l-k+1)$. Let us prove that $mv(H)-ne(H)>0$, and so $m/n>\rho(H)$.
\begin{align*}
mv(H)-ne(H) = m(e(H)s-(e(H)-1)-(l-k+1))-(m(s-1)-r)e(H) = \\
= m-m(l-k+1)+re(H) = -m(l-k)+r([lm/r]-[km/r]+2)> \\
> -m(l-k)+r(lm/r-km/r+1) = r\geq 0.
\end{align*}

The inequality is proved. Thus $G$ is strictly balanced with $\rho(G)=m/n$. 

\subsection{Case $\rho<\frac{1}{s-1}$}
\label{case2}

Let $G$ be an arbitrary $s$--hypergraph. Let $v_1,\ldots,v_{m}\in V(G)$, $e_1,\ldots,e_{m}\in E(G)$. We say that the sequence $(v_1,e_1,v_2,e_2,\ldots,v_{m},e_{m},v_{m+1}=v_1)$ is a {\it cycle} if for any $i=1,\ldots,m$
\begin{itemize}
\item $v_{i}\neq v_{i+1}$
\item $e_{i}\neq e_{i+1}$, where $e_{m+1}=e_1$
\item $\{v_{i},v_{i+1}\}\in e_{i}$.
\end{itemize}
We say that $s$--hypergraph is a {\it tree} if it is connected and has no cycles. It is easy to see that a tree with $k$ edges has exactly $1+k(s-1)$ vertices, and so its density is $\frac{k}{1+k(s-1)}$. Let $G$ be an $s$--hypergraph with a density less than $\frac{1}{s-1}$. Our goal is to prove that $G$ is a strictly balanced if and only if $G$ is a tree.

Note that every connected nonempty sub-hypergraph of a tree is a tree itself. Hence if $H$ is a connected proper sub-hypergraph of $G$ then
$$
 \rho(H)=\frac{e(H)}{1+e(H)(s-1)}<\frac{e(G)}{1+e(G)(s-1)}=\rho(G).
$$
Thus $G$ is strictly balanced.

Now, let $G$ be a strictly balanced $s$--hypergraph with a density $\rho<\frac{1}{s-1}$. Obviously, $G$ is connected. Note that if $G$ has a cycle $(v_{1},e_{1},v_{2},e_{2},\ldots,v_{n},e_{m},v_1)$ then the induced $s$--hypergraph $G|_{e_1\cup\ldots\cup e_{m}}$ has at least $m$ edges and at most $m(s-1)$ vertices, so its density is not less than $\frac{1}{s-1}$. So $G$ is acyclic and connected, hence $G$ is a tree.

\section{Main tools}
\label{main_tools}
In order to prove Theorem~\ref{irrational} we use the methods introduced by J.H. Spencer (see~\cite{strange_logic},~\cite{probabilistic_method}). In this section we generalize some definitions and theorems from~\cite{strange_logic},~\cite{counting_ext}. The results obtained in this section will be used in the proof of Theorem~\ref{irrational} in Section~\ref{proof_irrational}.

\subsection{Extensions}
\label{extensions}

In this subsection we generalize some well known definitions and review some theorems from~\cite{strange_logic},~\cite{counting_ext}, that are concerned with rooted graphs and their extensions.

Let $s\geq 2$ and irrational $\alpha>0$ be fixed throughout this section. A {\it rooted $s$-hypergraph} is a pair $(R,H)$ where $H$ is an $s$-hypergraph on vertex set, say, $V(H)=\{a_1,\ldots,a_r,b_1,\ldots,b_v\}$ and $R=\{a_1,\ldots,a_r\}\subsetneq V(H)$ is a specified subset of $V(H)$. Note that $R$ could be empty. The vertices of $R$ are called {\it roots}. Let $v=v(R,H)$ denote the number of vertices that are not roots and let $e=e(R,H)$ denote the number of edges of $H$ excluding those edges with all $s$ vertices being roots. We call $(v,e)$ the {\it type} of $(R,H)$. We say that $(R,H)$ is {\it dense} if $v-e\alpha<0$ and {\it sparse} if $v-e\alpha>0$. As $\alpha$ is irrational, then every $(R,H)$ is either dense or sparse. If $R\subseteq S\subset V(H)$, then we call $(R,H|_{S})$ a {\it subextension} of $(R,H)$. If $R\subset S\subseteq V(H)$ we call $(S,H)$ a {\it nailextension} of $(R,H)$. We call $(R,H)$ {\it rigid} if all of its nailextensions are dense. We call $(R,H)$ {\it safe} if all of its subextensions are sparse. We call $(R,H)$ {\it minimally safe} if it is safe and has no safe nailextensions. We sometimes write $(R,S)$ for $(R,H|_{S})$ when the hypergraph $H$ is understood.

Review some properties of rooted hypergraphs (their proofs in the case $s=2$ can be found in~\cite{strange_logic}, for arbitrary $s$ the proofs are the same).

\begin{claim}
\label{has_rigid}
Let $(R,H)$ be not safe. Then it has a rigid subextension.
\end{claim}

\begin{claim}
\label{has_safe}
Let $(R,H)$ be not rigid. Then it has a safe nailextension.
\end{claim}

\begin{claim}
\label{minimally_safe}
If $(R,H)$ is minimally safe and $R\subset S\subset V(H)$ then $(S,H)$ is rigid.
\end{claim}

\begin{claim}
\label{less_than_one}
Let $(R,H)$ be minimally safe of type $(v,e)$ with $v>1$. Then $v-e\alpha<1$.
\end{claim}

Let $G$ be an arbitrary $s$-hypergraph. Let ${\bf x}=(x_1,\ldots,x_r)$ be an $r$-tuple of distinct vertices of $G$ and let ${\bf y}=(y_1,\ldots,y_v)$ be a $v$-tuple of distinct vertices of $G$ such that $\x \cap \y = \varnothing$. We say that $\y$ is an {\it $(R,H)$-extension} of $\x$ if $\{x_{i_1},\ldots,x_{i_k},y_{i_{k+1}},\ldots,y_{i_s}\}\in E(G)$ whenever $\{a_{i_1},\ldots,a_{i_k},b_{i_{k+1}},\ldots,b_{i_s}\}\in E(H)$ where $0\leq k<s$. Note that $G|_{\x\cup\y}$ should have all edges of $H$ which contain at least one nonroot, but there are no restrictions for containing some additional edges. If there are no additional edges, then we say that $\y$ is an {\it exact $(R,H)$-extension} of $\x$. For any $\x=(x_1,\ldots,x_r)$ of distinct vertices in $G^{s}(n,p)$ let $N_{\x}$ denote the number of $(R,H)$-extensions $\y$ of $\x$.

Review the results on the distribution of the number of extensions in $G^{s}(n,p)$ (their proofs in the case $s=2$ can also be found in~\cite{strange_logic}, for arbitrary $s$ the proofs are the same; the proofs are based on the properties of rooted hypergraphs and Janson's Inequality (see~\cite{Janson})).

\begin{claim}
\label{exp_small}
Let $(R,H)$ be minimally safe and have type $(v,e)$ with $r$ roots. Let $\x$ be a fixed $r$-tuple from $G^{s}(n,p)$ with $p=n^{-\alpha}$. Set $\mu=\mathsf{E}N_{\x}$. Then $\Pr[N_{\x}=0]=e^{-\mu(1+o(1))}$.
\end{claim}

The proof of the following theorem is based on Claim~\ref{exp_small}.

\begin{theorem} [{\bf Counting Extensions Theorem}]
\label{counting_extension}
Let $(R,H)$ be safe and have type $(v,e)$ with $r$ roots. Then almost surely
$$
 N_{\x}\sim\mu
$$
for all $\x=(x_1,\ldots,x_r)$ of distinct vertices in $G^{s}(n,n^{-\alpha})$.
\end{theorem}

\subsection{The $t$-closure}
\label{t_closure}
Let $G$ be an arbitrary $s$-hypergraph and let {\it irrational} $\alpha>0$ be fixed. A {\it rigid $t$-chain} in $G$ is a sequence $\x=\x_0\subset \x_1\subset\ldots\subset\x_{k}$ of subsets of $V(G)$ with all $(\x_{i-1},G|_{\x_{i}})$ rigid (with respect to $\alpha$) and all $|\x_i\setminus\x_{i-1}|\leq t$. The {\it $t$-closure of $\x$}, denoted by $\cl_{t}(\x)$, is the maximal $\y$ for which there exists a rigid $t$-chain (of arbitrary length) $\x=\x_0\subset\x_1\subset\ldots\subset\x_{k}=\y$. When there are no such rigid $t$-chains we define $\cl_{t}(\x)=\x$. To see that $t$-closure is well defined we note that if $\x=\x_0\subset\x_1\subset\ldots\subset\x_{k}=\z$ and $\y=\y_0\subset\y_1\subset\ldots\subset\y_{l}=\y$ are rigid $t$-chains then so is $\x=\x_0\subset\x_1\subset\ldots\subset\x_{k}\cup\y_1\subset\ldots\subset\x_{k}\cup\y_{l}=\z\cup\y$. Alternatively, the $t$-closure $\cl_{t}(\x)$ is the minimal set containing $\x$ which has no rigid extensions of $\leq t$ vertices. Let $G$, $G'$ be $s$-hypergraphs with $r$-tuples $\x=(x_1,\ldots,x_r)$, $\y=(y_1,\ldots,y_r)$ respectively. We say that {\it $\x$, $\y$ have the same $t$-type} if there exists an isomorphism between their $t$-closures  sending each $x_{i}$ to its corresponding $y_{i}$.

The following result is a corollary of Theorem~\ref{small_distribution} (its proof in the case $s=2$ can be found in~\cite{strange_logic},~\cite{probabilistic_method}, for arbitrary $s$ the proof is the same).

\begin{theorem} [{\bf Finite Closure Theorem}]
\label{finite_closure}
Fix positive integers $t$ and $r$. Set $\varepsilon$ equal to the minimal value of $(e\alpha-v)/v$ over all integers $v$, $e$ with $1\leq v\leq t$ and $e\alpha-v>0$. Let $K$ be such that $r-K\varepsilon<0$. Then in $G^{s}(n,n^{-\alpha})$ almost surely
$$
 |\cl_{t}(X)|\leq K+r
$$
for all $X\subset V(G)$ with $|X|=r$.
\end{theorem}

\subsection{Generic extension}
Let $\alpha>0$ be fixed, let $(R,H)$ be a rooted graph, and let $t$ be a positive integer. We say that $\y=(y_1,\ldots,y_{v})$ is a {\it $t$-generic $(R,H)$-extension} of $\x$ if the following two properties hold.
\begin{itemize}
\item $\y$ is an exact $(R,H)$-extension of $\x$.
\item If any $\z=(z_1,\ldots,z_{s})$ with $s\leq t$ forms a rigid extension over $\x\cup\y$ then there are no edges in $G|_{\x\cup\y\cup\z}$ containing at least one vertex from $\y$ and at least one vertex from $\z$. 
\end{itemize} 

The next theorem (its proof for graphs can be found in~\cite{strange_logic}) which is a corollary of Theorems~\ref{counting_extension},~\ref{finite_closure} and properties of rooted hypergraphs (see Section~\ref{extensions}) states the existence of a generic extension. 

\begin{theorem} [{\bf Generic Extension Theorem}]
\label{generic_extension}
If $(R,H)$ is safe, relative to $\alpha$, then in $G\sim G^{s}(n,p)$ with $p=n^{-\alpha}$ almost surely every $\x$ has a $t$-generic extension $\y$.
\end{theorem}

\subsection{Ehrenfeucht game}
Let us define the {\it Ehrenfeucht Game $\EHR(G_1,G_2;k)$} with two disjoint $s$-hypergraphs $G_1$ and $G_2$ ($V(G_1)\cap V(G_2)=\varnothing$), two players (Spoiler and Duplicator) and a fixed number of rounds $k$. Each round has two parts, Spoiler's move followed by Duplicator's move. On the $i$-th round Spoiler selects either a vertex $x_i\in V(G_1)$ or a vertex $y_i\in V(G_2)$. Then Duplicator must select a vertex from the other graph. If Spoiler chooses previously chosen vertex, say, $x_i=x_j\in V(G_1)$ $(j<i)$, then Duplicator must choose the corresponding vertex $y_j\in V(G_2)$. If Spoiler chooses a vertex $x_i\notin\{x_1,\ldots,x_{i-1}\}$, then Duplicator must choose a vertex $y_i\notin\{y_1,\ldots,y_{i-1}\}$. If Duplicator can not do it then Spoiler wins the game. At the end of the game the vertices $x_1,x_2,\ldots,x_{k}\in V(G_1)$ and $y_1,y_2,\ldots,y_{k}\in V(G_2)$ are chosen. For Duplicator to win she must assure that, for all $1\leq i\leq j\leq k$, $x_i$, $x_j$ are adjacent if and only if $y_i$, $y_j$ are adjacent. If Duplicator does not win then Spoiler wins.

\begin{theorem}
\label{bridge_theorem}
The random $s$-hypergraph $G^{s}(n,p(n))$ obeys Zero-One Law if and only if for every positive integer $k$ 
$$
 \lim\limits_{n,m\rightarrow\infty}\Pr[\text{Duplicator wins } \EHR(G_1,G_2;k)]=1,
$$
where $G_1\sim G^{s}(n,p(n))$, $G_2\sim G^{s}(m,p(m))$ are independently chosen and have disjoint vertex sets.
\end{theorem}

\section{Proof of Theorem~\ref{irrational}}
\label{proof_irrational}
Let us prove Theorem~\ref{irrational}. We fix an irrational $\alpha>0$. All probabilities are with respect to the random $s$-hypergraph $G^{s}(n,p)$ with $p=n^{-\alpha}$.

Our approach is through the Ehrenfeucht game. We fix the number $k$ of moves. We will give a strategy for Duplicator so that, as $n,m\rightarrow\infty$, she almost surely wins $\EHR(G_1,G_2,k)$ where $G_1\sim G^{s}(n,n^{-\alpha})$ and $G_2\sim G^{s}(m,m^{-\alpha})$ are independently chosen.

Let $a_1,a_2,\ldots,a_k=0$ be nonnegative integers. Duplicator uses $(a_1,\ldots,a_k)$-{\it look-ahead strategy} which is defined in the following way. Duplicator makes any moves in response to Spoiler so that at the end of $r$-th move the $a_r$-types of the vertices chosen are the same in both hypergraphs. That is, if $x_1,\ldots,x_r\in G_1$, $y_1,\ldots,y_r\in G_2$ have been chosen then there is a hypergraph isomorphism from $\cl_{a_r}(x_1,\ldots,x_r)$ to $\cl_{a_r}(y_1,\ldots,y_r)$ sending each $x_j$ to its corresponding $y_j$. Note that if Duplicator is able to keep to this strategy then at the end of the game the 0-closures are the same and she has won. Following~\cite{strange_logic}, we define define $a_1$,\dots,$a_k$ by reverse induction so that Duplicator will almost surely be able to keep to $(a_1,\ldots,a_k)$--look-ahead strategy. Suppose, inductively, that $a_{r+1}$ has been defined. We define $a_r$ to be any integer satisfying
\begin{itemize}
\item[1.] $a_{r}\geq a_{r+1}$,
\item[2.] almost surely $|\cl_{a_{r+1}}(W)|-r\leq a_{r}$ for all sets $W$ of size $r+1$.
\end{itemize}
The existence of such $a$ is a consequence of Theorem~\ref{finite_closure}. In~\cite{strange_logic} it is proved that almost surely this strategy works. In the original proof, only graphs are considered. However, this proof is based on Theorem~\ref{generic_extension} only, which is also true for hypergraphs. So, by Theorem~\ref{bridge_theorem} $G^{s}(n,n^{-\alpha})$ obeys Zero-One Law when $\alpha$ is a positive irrational number.

\end{document}